\documentclass[11pt,reqno]{amsart}
\usepackage{amsmath,amsfonts,amssymb,amsthm,amscd,comment,euscript}
\usepackage[all]{xy}
\usepackage{graphicx}
\usepackage{mathptmx}
\usepackage{enumerate} %allows to change enumeration items easily
\usepackage[colorlinks=true]{hyperref} %to generate pdf with links
\usepackage[usenames,dvipsnames]{xcolor}
\usepackage{enumitem}

\xymatrixcolsep{1.9pc}                          % Adjust size of diagrams.
\xymatrixrowsep{1.9pc}
\newdir{ >}{{}*!/-5pt/\dir{>}}                  % Make better tailed arrows

 \calclayout

\swapnumbers
\theoremstyle{plain}
\newtheorem{lem}{Lemma}[section]
\newtheorem{cor}[lem]{Corollary}
\newtheorem{prop}[lem]{Proposition}
\newtheorem{thm}[lem]{Theorem}
\newtheorem*{thmm}{Theorem}

\theoremstyle{definition}
\newtheorem{ex}[lem]{Example}
\newtheorem{rem}[lem]{Remark}
\newtheorem{dfn}[lem]{Definition}

\renewcommand{\phi}{\varphi}
\renewcommand{\leq}{\leqslant}
\renewcommand{\geq}{\geqslant}
\renewcommand{\epsilon}{\varepsilon}

\renewcommand{\kappa}{\varkappa}

\DeclareMathOperator{\spec}{Spec}

 \DeclareMathOperator{\charr}{char}

 \DeclareMathOperator{\Map}{Map}

\DeclareMathOperator{\Hom}{Hom} 
 \DeclareMathOperator{\id}{id}

 \DeclareMathOperator{\colim}{colim}

 \DeclareMathOperator{\kr}{Ker}
 
 \DeclareMathOperator{\nis}{nis}

 \DeclareMathOperator{\Ob}{Ob}

\newcommand{\lra}[1]{\bl{#1}\longrightarrow\relax}
\newcommand{\bl}[1]{\buildrel #1\over}
\newcommand{\cc}{\mathcal}
\newcommand{\bb}{\mathbb}

\newcommand{\op}{{\textrm{\rm op}}}

\newcommand{\wh}{\widehat}
\newcommand{\wt}{\widetilde}

\newcommand{\gmp}{\bb G_m^{\wedge 1}}
\newcommand{\gmpn}{\bb G_m^{\wedge n}}
\newcommand{\uhom}{\underline{\Hom}}

\begin{document}

\footskip30pt

%\baselineskip=1.1\baselineskip

\title{Group schemes and motivic spectra}
\author{Grigory Garkusha}
\address{Department of Mathematics, Swansea University, Fabian Way, Swansea SA1 8EN, UK}
\email{g.garkusha@swansea.ac.uk}

\thanks{Supported by EPSRC grant EP/W012030/1}

%\date{29 November, 2018}

\begin{abstract}
By a theorem of Mandell, May, Schwede and Shipley~\cite{MMSS} the stable
homotopy theory of classical $S^1$-spectra is recovered from orthogonal spectra.
In this paper general linear, special linear, symplectic, orthogonal and special
orthogonal motivic spectra are introduced and studied. It is shown that
stable homotopy theory of motivic spectra is recovered from each of these types of spectra.
%An application is given for semistable symmetric motivic spectra.
An application is given for
the localization functor $C_*\cc Fr:SH_{\nis}(k)\to SH_{\nis}(k)$
in the sense of~\cite{GP5} that converts Morel--Voevodsky stable motivic homotopy theory $SH(k)$ into
the equivalent local theory of framed bispectra~\cite{GP5}.
\end{abstract}



\maketitle

\thispagestyle{empty} \pagestyle{plain}

\newdir{ >}{{}*!/-6pt/@{>}}

%\tableofcontents

\section{Introduction}

In the 90's several approaches to the stable homotopy theory of $S^1$-spectra
were suggested. In~\cite{MMSS} several comparison theorems relating the different constructions
were proven showing that all of the known approaches to highly structured ring
and module spectra are essentially equivalent.

Mandell, May, Schwede and Shipley~\cite{MMSS} proved that the stable
homotopy theory of classical topological $S^1$-spectra is recovered from orthogonal spectra.
In~\cite{Ost} {\O}stv{\ae}r conjectured that the stable
homotopy theory of motivic spectra can be recovered from motivic GL-spectra, in which the role
of the orthogonal groups as in topology~\cite{MMSS} is played by the general linear group
schemes $GL_n$-s. In this paper this conjecture is solved in the affirmative.

We follow~\cite{MMSS} to develop the formal theory of diagram motivic spectra
in Section~\ref{tutu}. The framework allows lots of flexibility
so that the reader can construct further interesting examples. For our purposes
we work with diagram motivic spectra coming from group schemes $GL_n$-s,
$SL_n$-s, $Sp_n$-s, $O_n$-s and $SO_n$-s (see Section~\ref{motspgrsch}).
These group schemes act on motivic spheres.
We also refer to the associated
motivic spectra as general linear, special linear, symplectic, orthogonal and special
orthogonal motivic spectra or just GL-, SL-, Sp-, O-, SO-motivic spectra.

One of the tricky concepts in the stable homotopy
theory of classical symmetric spectra is that of semistability. Semistable symmetric
spectra are important for understanding the difference between stable equivalences and maps
inducing $\pi_*$-isomorphisms, that is, isomorphisms of the classical stable homotopy groups
(in contrast with most other categories of spectra, not all stable
equivalences of symmetric spectra induce $\pi_*$-isomorphisms). The
same concept of semistability occurs in the stable homotopy theory
of motivic spectra. We show in Section~\ref{semsection} that every GL-, SL-
or Sp-motivic spectrum is semistable ragarded as a symmetric motivic spectrum.
This fact is the motivic counterpart of the classical result in topology
saying that every orthogonal $S^1$-spectrum of topological spaces
is semistable.

We then define in Section~\ref{modelstr} stable model structures on
the categories of diagram motivic spectra.
The main result of the paper is proven in Section~\ref{compthmsection}
which compares ordinary/symmetric motivic spectra with GL-, SL-,
Sp-, O- and SO-motivic spectra respectively (cf.
Mandell--May--Schwede--Shipley~\cite[0.1]{MMSS}).

\begin{thmm}[Comparison]
Let $k$ be any field.
The following natural adjunctions between categories of $T$- and
$T^2$-spectra are all Quillen equivalences with respect to the
stable model structure:
\begin{itemize}
\item[$(1)$]
   $Sp_{T}(k)\rightleftarrows Sp_{T}^{\textrm{GL}}(k)$;
\item[$(2)$]
   $Sp_{T^2}(k)\rightleftarrows Sp_{T^2}^{\textrm{SL}}(k)$;
\item[$(3)$]
   $Sp_{T^2}(k)\rightleftarrows Sp_{T^2}^{\textrm{Sp}}(k)$;
\item[$(4)$]
   $Sp_{T^2}(k)\rightleftarrows Sp_{T^2}^{\textrm{SO}}(k)\rightleftarrows Sp_{T^2}^{\textrm{O}}(k)$
if $\charr k\ne2$.
\end{itemize}
\end{thmm}

An application of the Comparison Theorem is given in
Section~\ref{locfun} for the localizing functor
   $$C_*\cc Fr:SH_{\nis}(k)\to SH_{\nis}(k)$$
in the sense of~\cite{GP5}. Recall that a new approach to the
classical Morel--Voevodsky stable homotopy theory $SH(k)$ was
suggested in~\cite{GP5} and is based on the functor $C_*\cc Fr$.
This approach has nothing to do with any kind of motivic
equivalences and is briefly defined as follows. We start with the
local stable homotopy category of sheaves of $S^1$-spectra
$SH_{S^1}^{\nis}(k)$. Then stabilizing $SH_{S^1}^{\nis}(k)$ with
respect to the endofunctor $\gmp\wedge-$, we arrive at the
triangulated category of bispectra $SH_{\nis}(k)$. We then apply an
{\it explicit\/} localization functor
   $$C_*\cc Fr:SH_{\nis}(k)\to SH_{\nis}(k)$$
that first takes a bispectrum $E$ to its naive projective cofibrant
resolution $E^c$ and then one sets in each bidegree $C_*\cc
Fr(E)_{i,j}:=C_*Fr(E^c_{i,j})$. The localization functor $C_*\cc Fr$
is isomorphic to the big framed motives localization functor $\cc
M^b_{fr}$ of~\cite{GP1} (see~\cite{GP5} as well). Then $SH^{new}(k)$
is defined as the category of $C_*\cc Fr$-local objects in
$SH_{\nis}(k)$. By~\cite[Section~2]{GP5} $SH^{new}(k)$ is
canonically equivalent to Morel--Voevodsky's $SH(k)$.

Using the Comparison Theorem above, we define new functors $C_*\cc
Fr^{\cc G,n}$ on $SH_{\nis}(k)$ that depend on $n\geq 0$ and the
choice of the family of groups $\cc G=\{GL_k\}_{k\geq
0},\{SL_{2k}\}_{k\geq 0},\{Sp_{2k}\}_{k\geq 0}$, $\{O_{2k}\}_{k\geq 0}$,
$\{SO_{2k}\}_{k\geq 0}$. In Theorem~\ref{prilozh} we prove that $C_*\cc
Fr$ and $C_*\cc Fr^{\cc G,n}$ are naturally isomorphic. As a result,
one can incorporate linear algebraic groups into the theory of
motivic infinite loop spaces and framed motives developed
in~\cite{GP1}.

Throughout the paper we denote by $S$ a Noetherian scheme of finite
dimension. We write $Sm/S$ for the category of smooth separated schemes of
finite type over $S$. $Sm/S$ comes equipped with the Nisnevich
topology~\cite[p.~95]{MV}. We denote by
$(Shv_\bullet(Sm/S),\wedge,pt_+)$ the closed symmetric monoidal
category of pointed Nisnevich sheaves on $Sm/S$. The category of
pointed motivic spaces $\bb M_\bullet$ is, by definition, the
category $\Delta^{\op}Shv_\bullet(Sm/k)$ of pointed simplicial Nisnevich
sheaves. Unless otherwise specified, we shall always deal with
the flasque local (respectively motivic) model structure on
$\bb M_\bullet$ in the sense of~\cite{Is}. Both model structures are
weakly finitely generated in the sense of~\cite{DRO}.

\subsubsection*{Acknowledgements}
The author is very grateful to Alexey Ananyevskiy, Semen Podkorytov and Matthias Wendt for
numerous helpful discussions. He also thanks Aravind Asok, Andrei Druzhinin
and Sergey Gorchinsky for various comments.

\section{Diagram motivic spaces and diagram motivic spectra}\label{tutu}

We refer the reader to~\cite{Bor} for basic facts of enriched category theory.
We mostly adhere to~\cite{MMSS} in this section.
Suppose $\cc C$ is a small category enriched over the closed
symmetric monoidal category of pointed motivic spaces $\bb
M_\bullet$. Following~\cite{MMSS} a {\it motivic $\cc C$-space\/} or
just a {\it $\cc C$-space\/} is an enriched functor $X:\cc C\to\bb
M_\bullet$. The category of motivic $\cc C$-spaces and $\bb
M_\bullet$-natural transformations between them is denoted by $[\cc
C,\bb M_\bullet]$. In the language of enriched category theory $[\cc
C,\bb M_\bullet]$ is the category of enriched functors from the $\bb
M_\bullet$-category $\cc C$ to the $\bb M_\bullet$-category $\bb
M_\bullet$. When $\cc C$ is enriched over unbased motivic spaces, we
implicitly adjoin a base object $*$; in other words, we then
understand $\cc C(a,b)$ to mean the union of the unbased motivic
space of maps from $a$ to $b$ in $\cc C$ and a disjoint basepoint.

\begin{dfn}\label{eval}
For an object $a\in\cc C$, define the {\it evaluation functor\/}
$Ev_a:[\cc C,\bb M_\bullet]\to\bb M_\bullet$ by $Ev_a(X)=X(a)$.
We also define the {\it shift desuspension functor\/} $F_a:\bb
M_\bullet\to[\cc C,\bb M_\bullet]$ by $F_a(A)=\cc C(a,-)\wedge A$
with $\cc C(a,-)$ the enriched functor represented by $a$. Then
$F_a$ is left adjoint to $Ev_a$.
\end{dfn}

For any $X\in[\cc C,\bb M_\bullet]$ there is a canonical isomorphism
   $$X\cong\int^{c\in\cc C}\cc C(c,-)\wedge X(c)=\int^{c\in\cc C}F_c(X(c)).$$
If $\cc C$ is a symmetric monoidal $\bb M_\bullet$-category with
monoidal product $\diamond$ and monoidal unit $u$, then $[\cc C,\bb M_\bullet]$ is a closed
symmetric monoidal $\bb M_\bullet$-category with monoidal product
   $$X\wedge Y=\int^{(a,b)\in\cc C\otimes\cc C}\cc C(a\diamond b,-)\wedge X(a)\wedge Y(b).$$
The monoidal unit is $\cc C(u,-)$. Moreover, $\cc C(a,-)\wedge\cc
C(b,-)\cong\cc C(a\diamond b,-)$. It follows that
   $$F_a(A)\wedge F_b(B)\cong F_{a\diamond b}(A\wedge B),\quad A,B\in\bb M_\bullet.$$

\begin{dfn}\label{spectra}
Suppose $(\cc C,\diamond,u)$ is a symmetric monoidal $\bb
M_\bullet$-category and $R$ is a ring object in $[\cc C,\bb
M_\bullet]$ with unit $\lambda$ and product $\phi$.
Following~\cite[1.9]{MMSS} a {\it $\cc C$-spectrum over $R$\/} is a
$\cc C$-space $X\in[\cc C,\bb M_\bullet]$ together with maps
$\sigma:X(a)\wedge R(b)\to X(a\diamond b)$, natural in $a$ and $b$,
such that the composite
   $$X(a)\cong X(a)\wedge S^0\xrightarrow{\id\wedge\lambda}X(a)\wedge R(u)\xrightarrow{\sigma}X(a\diamond u)\cong X(a),$$
where $S^0:=pt_+$,
is the identity and the following diagram commutes:
   $$\xymatrix{X(a)\wedge R(b)\wedge R(c)\ar[r]^{\sigma\wedge\id}\ar[d]_{\id\wedge\phi}&X(a\diamond b)\wedge R(c)\ar[d]^\sigma\\
                       X(a)\wedge R(b\diamond c)\ar[r]^\sigma&X(a\diamond b\diamond c).}$$
The category of $\cc C$-spectra over $R$ is denoted by $[\cc C,\bb
M_\bullet]_R$. It is tensored and cotensored over $\bb M_\bullet$.
\end{dfn}

The following lemma is straightforward.

\begin{lem}\label{monoid}
Suppose $\cc C$ is a symmetric monoidal $\bb M_\bullet$-category and $R$ is a ring
object in $[\cc C,\bb M_\bullet]$. Then the categories of (right) $R$-modules
and of $\cc C$-spectra over $R$ are isomorphic.
\end{lem}

A theorem of Day~\cite{Day} also implies the following

\begin{lem}\label{commonoid}
Let $\cc C$ be a symmetric monoidal $\bb M_\bullet$-category and $R$
a commutative ring object in $[\cc C,\bb M_\bullet]$. Then the
category of $R$-modules $[\cc C,\bb M_\bullet]_R$ has a smash
product $\wedge_R$ and internal Hom-functor $\uhom_R$ under which it
is a closed symmetric monoidal category with unit $R$.
\end{lem}

Let $\cc C$ be a symmetric monoidal $\bb M_\bullet$-category and $R$
a (not necessarily commutative) ring object in $[\cc C,\bb
M_\bullet]$. Mandell, May, Schwede and
Shipley~\cite[Section~2]{MMSS} suggested another description of the
category of $\cc C$-spaces over $R$. Namely, $[\cc C,\bb
M_\bullet]_R$ can be identified with the category of $\cc
C_R$-spaces, where
   $$\cc C_R(a,b):=[\cc C,\bb M_\bullet]_R(\cc C(b,-)\wedge R,\cc C(a,-)\wedge R).$$
The right hand side refers to the $\bb M_\bullet$-object in $[\cc
C,\bb M_\bullet]_R$. Composition is inherited from composition in
$[\cc C,\bb M_\bullet]_R$. Thus $\cc C_R$ can be regarded as the
full $\bb M_\bullet$-subcategory of $[\cc C,\bb M_\bullet]_R^{\op}$
whose objects are the free $R$-modules $\cc C(a,-)\wedge R$. By
construction,
   \begin{gather*}
    \cc C_R(a,b)=[\cc C,\bb M_\bullet]_R(\cc C(b,-)\wedge R,\cc C(a,-)\wedge R)\cong
       [\cc C,\bb M_\bullet](\cc C(b,-),\cc C(a,-)\wedge R)\cong\\
       (\cc C(a,-)\wedge R)(b)\cong\int^{(f,g)\in\cc C\otimes\cc C}\cc C(f\diamond g,b)\wedge\cc C(a,f)\wedge R(g).
   \end{gather*}

If $R$ is commutative, then $\cc C_R$ is symmetric monoidal with
monoidal product $\diamond_R$ on objects being defined as the
monoidal product $\diamond$ in $\cc C$. Its unit object is the unit
object $u$ of $\cc C$. The product $f\diamond_R f'$ of morphisms
$f:\cc C(b,-)\wedge R\to\cc C(a,-)\wedge R$ and $f':\cc
C(b',-)\wedge R\to\cc C(a',-)\wedge R$ is
   \begin{gather*}
    f\diamond_R f':\cc C(b\diamond b',-)\wedge R\cong(\cc C(b,-)\wedge R)\wedge_R(\cc C(b',-)\wedge R)\to\\
       (\cc C(a,-)\wedge R)\wedge_R(\cc C(a',-)\wedge R)\cong\cc C(a\diamond a',-)\wedge R.
   \end{gather*}

The proof of the following fact literally repeats that
of~\cite[2.2]{MMSS}, which is purely categorical and is not
restricted by topological categories only.

\begin{thm}[Mandell--May--Schwede--Shipley]\label{mmss}
Let $\cc C$ be a symmetric monoidal $\bb M_\bullet$-catego\-ry and
$R$ a ring object in $[\cc C,\bb M_\bullet]$. Then the categories
$[\cc C,\bb M_\bullet]_R$ of $\cc C$-spectra over $R$ and $[\cc
C_R,\bb M_\bullet]$ of motivic $\cc C_R$-spaces are isomorphic. If
$R$ is commutative, then the isomorphism $[\cc C,\bb
M_\bullet]_R\cong[\cc C_R,\bb M_\bullet]$ is an isomorphism of
symmetric monoidal categories.
\end{thm}

\section{Motivic spectra associated with group schemes}\label{motspgrsch}

After collecting basic facts for $\cc C$-spectra over a ring object
$R$ in $[\cc C,\bb M_\bullet]$, where $\cc C$ is a symmetric
monoidal $\bb M_\bullet$-category, in this section we give
particular examples we shall work with in this paper. The framework
we have fixed above allows a lot of flexibility and we invite the
interested reader to construct further examples. A canonical choice
for a ring object, which we denote by $\cc S$ or by $\cc S_{\cc C}$
if we want to specify the choice of the diagram $\bb
M_\bullet$-category $\cc C$, is the motivic sphere spectrum
   $$\cc S=(S^0,T,T^2,\ldots),$$
where $T^n$ is the Nisnevich sheaf $\bb A^n_S/(\bb A^n_S-0)$. Another natural choice
 is the motivic sphere $T^2$-spectrum
    $$\cc S=(S^0,T^2,T^4,\ldots)$$
consisting of the even dimensional spheres $T^{2n}$. The latter
spectrum is necessary below when working with, say, special linear
or symplectic groups. From the homotopy theory viewpoint, stable
homotopy categories of motivic $T$- and $T^2$-spectra are Quillen
equivalent (see, e.g.,~\cite[3.2]{PW}). Where it is possible we
follow the terminology and notation of~\cite{MMSS} in order to be
consistent with the classical topological examples.

We should stress that in all our examples below the category of
diagrams $\cc C$ is defined in terms of group schemes. Our first
example is elementary, but most important for our analysis.

\begin{ex}[Ordinary motivic $T$-spectra]\label{ordsp}
Let $\cc N$ be the (unbased) category of non-negative integers $\bb Z_{\geq 0}$,
with only ``identity morphisms motivic spaces" between them. Precisely,
  \[\cc N(m,n)=
  \begin{cases}
  pt, & m=n \\
  \emptyset, & m\not= n
  \end{cases}\]
The symmetric monoidal structure is given by addition $m+n$, with 0
as unit. An $\cc N$-space is a sequence of based motivic spaces. The
canonical enriched functor $\cc S=\cc S_{\cc N}$ takes $n\in\bb
Z_{n\geq 0}$ to $T^n$. It is a ring object of $[\cc N,\bb
M_\bullet]$, but it is not commutative since permutations of motivic
spheres $T^n$ are not identity maps. This is a typical difficulty in
defining the smash product in stable homotopy theory. A {\it motivic
$T$-spectrum\/} is an $\cc N$-spectrum over $\cc S$. Let $Sp_T^{\cc
N}(S)$ denote the category of $\cc N$-spectra over $\cc S$. Since
$T^n$ is the $n$-fold smash product of $T$, the category $Sp_T^{\cc
N}(S)$ is isomorphic to the category of ordinary motivic $T$-spectra
$Sp_T(S)$.

The shift desuspension functors to $\cc N$-spectra are given by $(F_mA)_n=A\wedge T^{n-m}$
(by definition, $T^{n-m}=*$ if $n<m$).
The smash product of $\cc N$-spaces (not $\cc N$-spectra!) is given by
   $$(X\wedge Y)_n=\bigvee_{i=0}^nX_i\wedge Y_{n-i}.$$
The category $\cc N_{\cc S}$ such that an $\cc N$-spectrum is an
$\cc N_{\cc S}$-space has morphism motivic spaces
   $$\cc N_{\cc S}(m,n)=T^{n-m}.$$
The category of ordinary motivic $T^2$-spectra $Sp^{\cc N}_{T^2}(S)$
is defined in a similar fashion.
\end{ex}

As we have noticed above, $\cc S_{\cc N}$ is not commutative, and
hence the category of $\cc N$-spectra $Sp_T^{\cc N}(S)$ does not
have a smash product that makes it a closed symmetric monoidal
category. In all other examples below the ring object $\cc S_{\cc
C}\in[\cc C,\bb M_\bullet]$ is commutative, and therefore the
category of $\cc C$-spectra over $\cc S_{\cc C}$ is closed symmetric
monoidal. The first classical example is that for symmetric spectra
(we refer the reader to~\cite{Jar} for further details).

\begin{ex}[Symmetric motivic $T$-spectra]\label{symsp}
Let $\Sigma$ be the (unbased) category of finite sets
$\mathbf{m}=\{1,\ldots,m\}$. By definition, $\mathbf{0}:=\emptyset$.
Its morphisms motivic spaces $\Sigma(\mathbf{m},\mathbf{n})$ are
given by symmetric groups canonically regarded as group $S$-schemes.
Precisely,
  \[\Sigma(\mathbf{m},\mathbf{n})=
  \begin{cases}
  \Sigma_m, & m=n \\
  \emptyset, & m\not= n
  \end{cases}\]
Notice that the underlying category associated with $\Sigma$ is $\bigsqcup_{n\geq 0}\Sigma_n$. The
symmetric monoidal structure on $\Sigma$ is given by concatenation
of sets $\mathbf{m}\sqcup\mathbf{n}$ and block sum of permutations,
with $\mathbf 0$ as unit. Commutativity of the monoidal product is
given by the shuffle permutation
$\chi_{m,n}:\mathbf{m}\sqcup\mathbf{n}\lra{\cong}\mathbf{n}\sqcup\mathbf{m}$
from the symmetric group $\Sigma_{m+n}$. The category $[\Sigma,\bb
M_\bullet]$ is isomorphic to the category of symmetric sequences of
pointed motivic spaces, i.e. the category of non-negatively graded
pointed motivic spaces with symmetric groups actions.

The canonical enriched functor $\cc S=\cc
S_{\Sigma}$ takes $\mathbf n$ to $T^n$ ($\Sigma_n$ permutes the $n$
copies of $T$ or, equivalently, the coordinates of $\bb A^n_S/(\bb
A^n_S-0)$). It is a  commutative ring object of $[\Sigma,\bb
M_\bullet]$. A {\it symmetric motivic $T$-spectrum\/} is a
$\Sigma$-spectrum over $\cc S$. Note that there is a canonical $\bb
M_\bullet$-functor $\iota:\cc N\to\Sigma$ mapping $n$ to $\mathbf n$
such that $\cc S_{\cc N}=\cc S_\Sigma\circ\iota$.

The shift desuspension functors to symmetric spectra are given by
   $$(F_mA)(\mathbf n)=\Sigma_{n+}\wedge_{\Sigma_{n-m}}(A\wedge T^{n-m}).$$
In turn, the smash product of $\Sigma$-spaces is given by
   $$(X\wedge Y)(\mathbf n)=\bigvee_{i=0}^n\Sigma_{n+}\wedge_{\Sigma_i\times\Sigma_{n-i}}
     X(\mathbf i)\wedge Y(\mathbf{n-i}).$$
The category $\Sigma_{\cc S}$ such that a $\Sigma$-spectrum is a $\Sigma_{\cc S}$-space
(see Theorem~\ref{mmss}) has morphism spaces
   $$\Sigma_{\cc S}(\mathbf{m},\mathbf{n})=\Sigma_{n+}\wedge_{\Sigma_{n-m}}T^{n-m}.$$
We shall write $Sp^\Sigma_T(S)$ to denote the category of symmetric
motivic $T$-spectra. The category of symmetric motivic $T^2$-spectra
$Sp^\Sigma_{T^2}(S)$ is defined in a similar fashion.
\end{ex}

\begin{ex}[GL-motivic $T$-spectra]\label{glsp}
Let $\textrm{GL}$ be the (unbased) category whose objects are the
non-negative integers $\bb Z_{\geq 0}$. Its morphisms motivic spaces
$\textrm{GL}(m,n)$ are given by the following group $S$-schemes:
  \[\textrm{GL}(m,n)=
  \begin{cases}
  GL_m, & m=n \\
  \emptyset, & m\not= n
  \end{cases}\]
The symmetric monoidal structure on $\textrm{GL}$ is given by
addition of integers and standard concatenation $GL_m\times GL_n\to
GL_{m+n}$ by block matrices. Commutativity of the monoidal product
is given by the shuffle permutation matrix $\chi_{m,n}\in GL_{m+n}$.
The canonical enriched functor $\cc S=\cc S_{\textrm{GL}}$ takes $n$
to $T^n$ ($GL_n$ acts on $T^n=\bb A^n_S/(\bb A^n_S-0)$ in a
canonical way). It is a commutative ring object of $[\textrm{GL},\bb
M_\bullet]$ because each $GL_n$ contains $\Sigma_n$ as permutation
matrices. A {\it GL-motivic $T$-spectrum\/} is a
$\textrm{GL}$-spectrum over $\cc S$. Note that there is a canonical
$\bb M_\bullet$-functor $\iota:\Sigma\to\textrm{GL}$ mapping
$\mathbf n$ to $n$ and mapping permutations to their permutation
matrices such that $\cc S_{\Sigma}=\cc S_{\textrm{GL}}\circ\iota$.

The shift desuspension functors to GL-spectra are given by the
induced motivic spaces (we refer the reader to~\cite{GM} for basic
facts on equivariant homotopy theory)
   \begin{equation}\label{desgl}
    (F_mA)(n)=(GL_{n})_+\wedge_{GL_{n-m}}(A\wedge T^{n-m}).
   \end{equation}
In turn, the smash product of GL-spaces is given by
   $$(X\wedge Y)(n)=\bigvee_{i=0}^n (GL_{n})_+\wedge_{GL_i\times GL_{n-i}}X(i)\wedge Y(n-i).$$
The category $\textrm{GL}_{\cc S}$ such that a GL-spectrum is a
$\textrm{GL}_{\cc S}$-space (see Theorem~\ref{mmss}) has morphism
spaces
   $$\textrm{GL}_{\cc S}(m,n)=(GL_{n})_+\wedge_{GL_{n-m}}T^{n-m}.$$
A typical example of a GL-spectrum is the algebraic cobordism
$T$-spectrum $MGL$ (this follows from~\cite[Section~4]{PW}). We
shall write $Sp^{\textrm{GL}}_T(S)$ to denote the category of
GL-motivic $T$-spectra.
\end{ex}

\begin{ex}[SL-motivic $T^2$-spectra]\label{slsp}
In contrast to general linear groups, special linear groups contain
only even permutations as their permutation matrices. We can equally
define the ``SL-category'' as in Example~\ref{glsp} whose objects
are all non-negative integers. The problem with such a $\bb
M_\bullet$-category of diagrams is that it is not symmetric monoidal
(unless characteristic is 2), and hence a problem with defining
corresponding ring objects. To fix the problem, we work with even
non-negative integers $2\bb Z_{\geq 0}$. We define morphisms motivic
spaces $\textrm{SL}(2m,2n)$ by the following group $S$-schemes:
  $$\textrm{SL}(2m,2n)=
  \begin{cases}
  SL_{2m}, & m=n \\
  \emptyset, & m\not= n
  \end{cases}$$
We use the embedding $i_n:\Sigma_n\hookrightarrow SL_{2n}$ taking
$\sigma\in\Sigma_n$ to the permutation matrix associated with
$\wt\sigma\in\Sigma_{2n}$, where $\wt\sigma(2i-1)=2\sigma(i)-1$ and
$\wt\sigma(2i)=2\sigma(i)$. With these embeddings of symmetric
groups into even-dimensional special linear groups the diagram
category SL becomes a symmetric monoidal $\bb M_\bullet$-category.
The symmetric monoidal structure on $\textrm{SL}$ is given by
addition of integers and standard concatenation $SL_{2m}\times
SL_{2n}\to SL_{2m+2n}$ by block matrices. Commutativity of the
monoidal product is given by the shuffle permutation matrix
$\chi_{2m,2n}=i_n(\chi_{m,n})\in SL_{2m+2n}$. The canonical enriched
functor $\cc S=\cc S_{\textrm{SL}}$ takes $2n$ to $T^{2n}$
($SL_{2n}$ acts on $T^{2n}=\bb A^{2n}_S/(\bb A^{2n}_S-0)$ in a
canonical way). It is a commutative ring object of $[\textrm{SL},\bb
M_\bullet]$ because each $SL_{2n}$ contains $\Sigma_n$ as
permutation matrices defined above. An {\it SL-motivic
$T^2$-spectrum\/} is an $\textrm{SL}$-spectrum over $\cc S$. Note
that there is a canonical $\bb M_\bullet$-functor
$\iota:\Sigma\to\textrm{SL}$ mapping $\mathbf n$ to $2n$ and
$\sigma\in\Sigma_n$ to $i_n(\sigma)$ such that the symmetric sphere
$T^2$-spectrum $(S^0,T^2,T^4,\ldots)$ equals $\cc
S_{\textrm{SL}}\circ\iota$. If there is no likelihood of confusion
we shall also denote the symmetric sphere $T^2$-spectrum
$(S^0,T^2,T^4,\ldots)$ by $\cc S_\Sigma$ whenever we work with
$T^2$-spectra. Notice that this $T^2$-spectrum $\cc S_\Sigma$ is a
commutative ring object of $[\Sigma,\bb M_\bullet]$ and the category
of right modules over $\cc S_\Sigma$ is isomorphic to the category
of symmetric $T^2$-spectra $Sp_{T^2}^\Sigma(S)$.

The shift desuspension functors to SL-spectra are given by the
induced motivic spaces
  \begin{equation}\label{dessl}
   (F_{2m}A)(2n)=(SL_{2n})_{+}\wedge_{SL_{2n-2m}}(A\wedge T^{2n-2m}).
  \end{equation}
In turn, the smash product of SL-spaces is given by
   $$(X\wedge Y)(2n)=\bigvee_{i=0}^n (SL_{2n})_{+}\wedge_{SL_{2i}\times SL_{2n-2i}}X(2i)\wedge Y(2n-2i).$$
The category $\textrm{SL}_{\cc S}$ such that an SL-spectrum is an
$\textrm{SL}_{\cc S}$-space (see Theorem~\ref{mmss}) has morphism
spaces
   $$\textrm{SL}_{\cc S}(2m,2n)=(SL_{2n})_{+}\wedge_{SL_{2n-2m}}T^{2n-2m}.$$
A typical example of an SL-spectrum is the algebraic special linear
cobordism $T^2$-spectrum $MSL$ in the sense of
Panin--Walter~\cite[Section~4]{PW}. We shall write
$Sp^{\textrm{SL}}_{T^2}(S)$ to denote the category of SL-motivic
$T^2$-spectra.
\end{ex}

\begin{ex}[Symplectic motivic $T^2$-spectra]\label{spsp}
Following~\cite[Section~6]{PW} we write the standard symplectic form
on the trivial vector bundle of rank $2n$ as
\[\omega_{2n}=\begin{bmatrix}
  0&1&{} & &{}\\
  -1&0&{} &0 &{}\\
   & {} & \ddots & {}\\
  &0&{} &0 &1\\
  &&{} &-1 &0\\
  \end{bmatrix}\]
The canonical symplectic isometry $(\cc O_S,\omega_{2n})\cong(\cc
O_S,\omega_{2})^{\oplus n}$ gives rise to a natural action of
$\Sigma_n$. It permutes the $n$ orthogonal direct summands, and
hence one gets an embedding $i_n:\Sigma_n\hookrightarrow Sp_{2n}$,
which sends permutations to the same permutation matrices as in
Example~\ref{slsp}. Let Sp have objects $2\bb Z_{\geq 0}$ and let
morphisms motivic spaces $\textrm{Sp}(2m,2n)$ be defined by the
following group $S$-schemes:
  \[\textrm{Sp}(2m,2n)=
  \begin{cases}
  Sp_{2m}, & m=n \\
  \emptyset, & m\not= n
  \end{cases}\]
With embeddings of symmetric groups into symplectic groups above the
diagram category Sp  becomes a symmetric monoidal $\bb
M_\bullet$-category. The symmetric monoidal structure on
$\textrm{Sp}$ is given by addition of integers and standard
concatenation $Sp_{2m}\times Sp_{2n}\to Sp_{2m+2n}$ by block
matrices. Commutativity of the monoidal product is given by the
shuffle permutation matrix $\chi_{2m,2n}\in Sp_{2m+2n}$. The
canonical  enriched functor $\cc S=\cc S_{\textrm{Sp}}$ takes $2n$
to $T^{2n}$ ($Sp_{2n}$ acts on $T^{2n}=\bb A^{2n}_S/(\bb
A^{2n}_S-0)$ in a canonical way). It is a commutative ring object of
$[\textrm{Sp},\bb M_\bullet]$ because each $Sp_{2n}$ contains
$\Sigma_n$ as permutation matrices defined above. A {\it symplectic
motivic $T^2$-spectrum\/} is an $\textrm{Sp}$-spectrum over $\cc S$.
Note that there is a canonical $\bb M_\bullet$-functor
$\iota:\Sigma\to\textrm{Sp}$ mapping $\mathbf n$ to $2n$ and
$\sigma\in\Sigma_n$ to $i_n(\sigma)$ such that the symmetric sphere
$T^2$-spectrum $\cc S_\Sigma=(S^0,T^2,T^4,\ldots)$ equals $\cc
S_{\textrm{Sp}}\circ\iota$.

The shift desuspension functors to symplectic spectra are given by
the induced motivic spaces
  \begin{equation}\label{dessp}
   (F_{2m}A)(2n)=(Sp_{2n})_{+}\wedge_{Sp_{2n-2m}}(A\wedge T^{2n-2m}).
  \end{equation}
In turn, the smash product of Sp-spaces is given by
   $$(X\wedge Y)(2n)=\bigvee_{i=0}^n (Sp_{2n})_{+}\wedge_{Sp_{2i}\times Sp_{2n-2i}}X(2i)\wedge Y(2n-2i).$$
The category $\textrm{Sp}_{\cc S}$ such that an Sp-spectrum is an
$\textrm{Sp}_{\cc S}$-space (see Theorem~\ref{mmss}) has morphism
spaces
   $$\textrm{Sp}_{\cc S}(2m,2n)=(Sp_{2n})_{+}\wedge_{Sp_{2n-2m}}T^{2n-2m}.$$
A typical example of a symplectic spectrum is the algebraic
symplectic cobordism $T^2$-spectrum $MSp$ in the sense of
Panin--Walter~\cite[Section~6]{PW}. We shall write
$Sp^{\textrm{Sp}}_{T^2}(S)$ to denote the category of symplectic
motivic $T^2$-spectra.
\end{ex}

In the next two examples we suppose $\frac12\in S$ and follow the
terminology and notation of~\cite{Con}. Denote by $q_{2m}$ the {\it
standard split quadratic form\/}
   $$q_{2m}=x_{1}x_{2}+x_{3}x_{4}+\cdots+x_{2m-1}x_{2m}.$$
We define $O_{2m}:=O(q_{2m})$ and $SO_{2m}:=SO(q_{2m})$.

\begin{ex}[Orthogonal motivic $T^2$-spectra]\label{orthsp}
Let O have objects $2\bb Z_{\geq 0}$ and let morphisms motivic
spaces $\textrm{O}(2m,2n)$ be defined by the following group
$S$-schemes:
  \[\textrm{O}(2m,2n)=
  \begin{cases}
  O_{2m}, & m=n \\
  \emptyset, & m\not= n
  \end{cases}\]
The corresponding embeddings of symmetric groups into orthogonal
groups are the same with those of Example~\ref{slsp}. Then the
diagram category O becomes a symmetric monoidal $\bb
M_\bullet$-category. The symmetric monoidal structure on
$\textrm{O}$ is given by addition of integers and standard
concatenation $O_{2m}\times O_{2n}\to O_{2m+2n}$ by block matrices.
Commutativity of the monoidal product is given by the shuffle
permutation matrix $\chi_{2m,2n}\in O_{2m+2n}$. The canonical
enriched functor $\cc S=\cc S_{\textrm{O}}$ takes $2n$ to $T^{2n}$
($O_{2n}$ acts on $T^{2n}=\bb A^{2n}_S/(\bb A^{2n}_S-0)$ in a
canonical way). It is a  commutative ring object of $[\textrm{O},\bb
M_\bullet]$ because each $O_{2n}$ contains $\Sigma_n$ as permutation
matrices defined above. An {\it orthogonal motivic $T^2$-spectrum\/}
is an $\textrm{O}$-spectrum over $\cc S$. Note that there is a
canonical $\bb M_\bullet$-functor $\iota:\Sigma\to\textrm{O}$
mapping $\mathbf n$ to $2n$ and $\sigma\in\Sigma_n$ to $i_n(\sigma)$
such that the symmetric sphere $T^2$-spectrum $\cc
S_\Sigma=(S^0,T^2,T^4,\ldots)$ equals $\cc
S_{\textrm{O}}\circ\iota$.

The shift desuspension functors to orthogonal spectra are given by
the induced motivic spaces
  \begin{equation}\label{deso}
   (F_{2m}A)(2n)=(O_{2n})_{+}\wedge_{O_{2n-2m}}(A\wedge T^{2n-2m}).
  \end{equation}
In turn, the smash product of O-spaces is given by
   $$(X\wedge Y)(2n)=\bigvee_{i=0}^n (O_{2n})_{+}\wedge_{O_{2i}\times O_{2n-2i}}X(2i)\wedge Y(2n-2i).$$
The category $\textrm{O}_{\cc S}$ such that an O-spectrum is an
$\textrm{O}_{\cc S}$-space (see Theorem~\ref{mmss}) has morphism
spaces
   $$\textrm{O}_{\cc S}(2m,2n)=(O_{2n})_{+}\wedge_{O_{2n-2m}}T^{2n-2m}.$$
We shall write $Sp^{\textrm{O}}_{T^2}(S)$ to denote the category of
orthogonal motivic $T^2$-spectra.
\end{ex}

\begin{ex}[SO-motivic $T^2$-spectra]\label{sosp}
The definition of this type of motivic $T^2$-spectra literally
repeats Example~\ref{orthsp} if we replace $O_{2n}$ with $SO_{2n}$
in all relevant places. The shift desuspension functors to
SO-spectra are given by the induced motivic spaces
  \begin{equation}\label{desso}
   (F_{2m}A)(2n)=(SO_{2n})_{+}\wedge_{SO_{2n-2m}}(A\wedge T^{2n-2m}).
  \end{equation}
The category $\textrm{SO}_{\cc S}$ such that an SO-spectrum is an
$\textrm{SO}_{\cc S}$-space (see Theorem~\ref{mmss}) has morphism
spaces
   $$\textrm{SO}_{\cc S}(2m,2n)=(SO_{2n})_{+}\wedge_{SO_{2n-2m}}T^{2n-2m}.$$
We shall write  $Sp^{\textrm{SO}}_{T^2}(S)$ to denote the category
of SO-motivic $T^2$-spectra.
\end{ex}

\section{Semistable motivic spectra}\label{semsection}

One of the tricky concepts in the stable homotopy
theory of classical symmetric spectra is that of semistability. The
same concept of semistability occurs in the stable homotopy theory
of motivic symmetric $T$- or $T^2$-spectra.

Namely, following R\"ondigs, Spitzweck and \O stv\ae r~\cite{RSO},
a motivic symmetric $T$-spectrum (likewise $T^2$-spectrum)
$E$ is said to be {\it semistable\/} if the natural map
   $$\phi(E):T\wedge E\to sh(E)$$
is a stable weak equivalence of underlying (non-symmetric) motivic spectra.
In level $n$ it is defined as the composite map
   $$T\wedge E_n\xrightarrow{\cong}E_n\wedge T\to E_{n+1}\xrightarrow{\chi_{n,1}}E_{1+n}$$
of the twist isomorphism, the $n$th structure map of the spectrum $E$ and the cyclic permutation
$\chi_{n,1}=(1,2,\ldots, n+1)$.

Similarly to the classical symmetric $S^1$-spectra (see, e.g.,~\cite[I.3.16]{Sch})
a motivic symmetric $T$- or $T^2$-spectrum
$X$ is semistable if for every $n$ and every even permutation
$\sigma\in\Sigma_n$ the action of $\sigma$ on $X_n$ coincides with the identity
in the pointed motivic unstable homotopy category~\cite[3.2]{RSO}.

It follows from Examples~\ref{glsp}-\ref{sosp} that every $G$-spectrum, where $G\in\{GL,SL,Sp\}$,
is a symmetric $T$- or $T^2$-spectrum.
It follows from~\cite[3.2]{Sch08} that every orthogonal $S^1$-spectrum of topological spaces
is semistable. The following theorem is a motivic counterpart of that fact.

\begin{thm}\label{semist}
Let $G\in\{GL,SL,Sp\}$. Then every $G$-spectrum is semistable as a symmetric $T$- or $T^2$-spectrum.
\end{thm}

\begin{proof}
GL-, SL- or Sp-motivic spectra have the property that
the action of the symmetric group $\Sigma_n$ on the motivic
spaces of GL-, SL- or Sp-motivic spectra factors through the action
of $GL_n$, $SL_{2n}$ and $Sp_{2n}$ respectively. Therefore, even
permutations are $\bb A^1$-homotopic to identity
(see~\cite[Section~2]{GN}).

In more detail, this means that if $E$ is a $G$-spectrum and
$\sigma\in\Sigma_n$ is an even permutation, then there is an $\bb A^1$-homotopy
$E_n\to\uhom(\bb A^1,E_n)$ between the action of $\sigma$ and the identity map.
It follows that the action of $\sigma$ on $E_n$ coincides with the identity
in the pointed motivic unstable homotopy category, and hence $E$
is semistable by~\cite[3.2]{RSO}.
\end{proof}

As a consequence of the preceding theorem, we get rid of the
semistability  phenomenon for GL-, SL- or Sp-motivic spectra.
Typical examples of such motivic spectra are $MGL$, $MSL$ and $MSp$.
It will follow from Theorem~\ref{compar} that symmetric motivic spectra are Quillen
equivalent to GL-, SL- or Sp-motivic spectra. Therefore we can make symmetric
motivic spectra GL-, SL- or Sp-motivic spectra by extending the
group action and then compute the latter spectra within GL-, SL- or
Sp-motivic spectra for which the phenomenon of semistability is irrelevant.

\section{Model structures for $\cc C$-spectra}\label{modelstr}

Throughout this section $\cc C$ is a small category of diagrams
enriched over $\bb M_\bullet$. Recall that $\bb M_\bullet$ is
equipped with the flasque motivic model structure in the sense
of~\cite{Is}. This model structure is simplicial, monoidal, proper,
cellular and weakly finitely generated in the sense of~\cite{DRO}.
It follows from~\cite[3.2.13]{MV} that the smash product preserves
motivic weak equivalences. Furthermore, $\bb M_\bullet$ satisfies
the monoid axiom in the sense of~\cite{SS}. In the flasque model structure every sheaf of the
form $X/U$ is cofibrant, where $U\hookrightarrow X$ is a
monomorphism in $Sm/S$. In particular, the sheaf $T^n$, $n\geq 0$,
is flasque cofibrant.

Following~\cite[Section~4]{DRO} $[\cc C,\bb M_\bullet]$ is equipped
with the pointwise model structure, where a map $f$ in $[\cc C,\bb
M_\bullet]$ is a {\it pointwise motivic weak equivalence\/}
(respectively a {\it pointwise fibration}) if $f(c)$ is a motivic
weak equivalence (respectively fibration) in $\bb M_\bullet$ for all
$c\in\Ob\cc C$. {\it Cofibrations\/} are defined as maps satisfying
the left lifting property with respect to all pointwise acyclic
fibrations.

\begin{prop}\label{pmodel}
The following statements are true:
\begin{itemize}
\item[$(1)$] $[\cc C,\bb M_\bullet]$ together with pointwise
fibrations, pointwise motivic equivalences
and cofibrations defined above is a simplicial cellular weakly
finitely generated $\bb M_\bullet$-model category.

\item[$(2)$] The pointwise model structure on $[\cc C,\bb M_\bullet]$ is proper.

\item[$(3)$] If $\cc C$ is a symmetric monoidal $\bb M_\bullet$-category, then $[\cc C,\bb M_\bullet]$
is a monoidal $\bb M_\bullet$-model category, and the monoid axiom in the sense of~\cite{SS} holds.
\end{itemize}
\end{prop}

\begin{proof}
(1). This follows from~\cite[4.2, 4.4]{DRO}.

(2). Since $\bb M_\bullet$ is right proper, then so is $[\cc C,\bb
M_\bullet]$ by~\cite[4.8]{DRO}. Furthermore, $\bb M_\bullet$ is
strongly left proper in the sense of~\cite[4.6]{DRO}.
By~\cite[4.8]{DRO} $[\cc C,\bb M_\bullet]$ is also left proper.

(3). This follows from~\cite[4.4]{DRO}.
\end{proof}

\begin{cor}\label{adjunc}
Let $\cc C$ be contained in a bigger $\bb M_\bullet$-category of
diagrams $\cc D$. Then the canonical adjunction
   $$L:[\cc C,\bb M_\bullet]\rightleftarrows[\cc D,\bb M_\bullet]:U,$$
where $L$ is the enriched left Kan extension and $U$ is the
forgetful functor, is a Quillen pair with respect to the pointwise
model structure.
\end{cor}

\begin{cor}\label{uhuh}
The categories of motivic $T$- and $T^2$-spectra $Sp^{\cc
N}_{T}(S)$, $Sp^{\cc N}_{T^2}(S)$, $Sp^{\Sigma}_{T}(S)$,
$Sp^{\Sigma}_{T^2}(S)$, $Sp^{\textrm{GL}}_{T}(S)$,
$Sp^{\textrm{SL}}_{T^2}(S)$, $Sp^{\textrm{O}}_{T^2}(S)$ and
$Sp^{\textrm{SO}}_{T^2}(S)$ of Examples~\ref{ordsp}-\ref{sosp} are
cellular weakly finitely generated proper $\bb M_\bullet$-model
categories. Moreover, $Sp^{\Sigma}_{T}(S)$, $Sp^{\Sigma}_{T^2}(S)$,
$Sp^{\textrm{GL}}_{T}(S)$, $Sp^{\textrm{SL}}_{T^2}(S)$,
$Sp^{\textrm{O}}_{T^2}(S)$ and $Sp^{\textrm{SO}}_{T^2}(S)$  are
monoidal $\bb M_\bullet$-model categories, and the monoid axiom
holds for them.
\end{cor}

\begin{proof}
This follows from Proposition~\ref{pmodel} and Theorem~\ref{mmss}.
\end{proof}

Recall that ordinary and symmetric motivic spectra have Quillen
equivalent stable model structures (see, e.g.,~\cite[4.31]{Jar}). We want
to extend the stable model structure further to diagram spectra of
Examples~\ref{glsp}-\ref{sosp}. To define it, we fix a symmetric
monoidal diagram $\bb M_\bullet$-category  $\cc C$ together with a
faithful strong symmetric monoidal functor of $\bb
M_\bullet$-categories $\iota:\Sigma\to\cc C$ and a sphere ring
spectrum $\cc S=\cc S_{\cc C}$ such that $\cc S_\Sigma=\cc S_{\cc
C}\circ\iota$. We shall always assume that $\cc S=(S^0,K,K^{\wedge
2},\ldots)$ with $K=T$ or $K=T^2$. By Theorem~\ref{mmss} we identify
the corresponding categories of spectra with categories
$[\Sigma_{\cc S},\bb M_\bullet]$ and $[\cc C_{\cc S},\bb
M_\bullet]$. As above, one has a natural adjunction
   $$L:[\Sigma_{\cc S},\bb M_\bullet]\rightleftarrows[\cc C_{\cc S},\bb M_\bullet]:U.$$

\begin{dfn}\label{modelst}
Following Hovey~\cite[8.7]{H}, define the {\it stable model
structure\/} on $[\cc C_{\cc S},\bb M_\bullet]$ to be the Bousfield
localization with respect to $\cc P$ of the pointwise model model
structure on $[\cc C_{\cc S},\bb M_\bullet]$, where
   $$\cc P=\{\lambda_n:F_{n+1}(C\wedge K)\to F_nC\}$$
as $C$ runs through the domains and codomains of the generating
cofibrations of $\bb M_\bullet$, and $n\geq 0$. The weak
equivalences of the model category $[\cc C_{\cc S},\bb M_\bullet]$
will be called {\it stable weak equivalences}. Note that if $\cc
C=\Sigma$ then the stable model structure is nothing but the
(flasque) stable model structure of symmetric spectra.
\end{dfn}

The preceding definition together with Corollary~\ref{adjunc} and~\cite[2.2]{H} imply
the following

\begin{prop}\label{adjuncst}
The canonical adjunction
   $$L:[\Sigma_{\cc S},\bb M_\bullet]\rightleftarrows[\cc C_{\cc S},\bb M_\bullet]:U,$$
where $L$ is the enriched left Kan extension and $U$ is the
forgetful functor, is a Quillen pair with respect to the stable
model structure.
\end{prop}

Since ordinary $T$- or $T^2$-spectra are Quillen equivalent to
symmetric spectra (see~\cite[4.31]{Jar}), the preceding proposition
implies the following

\begin{cor}\label{adjuncstcor}
The canonical adjunction
   $$L:[\cc N_{\cc S},\bb M_\bullet]\rightleftarrows[\cc C_{\cc S},\bb M_\bullet]:U,$$
where $L$ is the enriched left Kan extension and $U$ is the
forgetful functor, is a Quillen pair with respect to the stable
model structure.
\end{cor}

The main goal of the paper is to show that the adjunction of the
previous proposition is a Quillen equivalence for $\cc C$ being GL,
SL, Sp, O and SO if we make a further assumption that the base
scheme $S$ is the spectrum $\spec
k$ of a field $k$. This is treated in the next section.

\section{The comparison theorem}\label{compthmsection}

Throughout this section $k$ is any field. We shall freely operate with various equivalent
models for $SH(k)$ like $T$-/$\bb P^1$-spectra or
$(S^1,\gmp)$-bispectra. It will always be clear which of the models
is used.

The natural Quillen equivalences
   $$Sp_{T}^{\cc N}(k)\rightleftarrows Sp_{T}^{\Sigma}(k),\quad
     Sp_{T^2}^{\cc N}(k)\rightleftarrows Sp_{T^2}^{\Sigma}(k)$$
between ordinary and symmetric motivic $T$- or $T^2$-spectra are
well-known (see, e.g.,~\cite[4.31]{Jar}). The purpose of this section
is to establish Quillen equivalences between spectra having a further
structure given by various families of group schemes. Namely, we are now
in a position to formulate the main result of the paper which compares
ordinary/symmetric motivic spectra with GL-, SL-, Sp-, O- and
SO-motivic spectra respectively (cf.
Mandell--May--Schwede--Shipley~\cite[0.1]{MMSS}).

\begin{thm}[Comparison]\label{compar}
%Let $S$ be the spectrum $\spec k$ of a perfect field $k$.
The following natural adjunctions between categories of $T$- and
$T^2$-spectra are all Quillen equivalences with respect to the
stable model structure of Definition~\ref{modelst}:
\begin{itemize}
\item[$(1)$]
   $Sp_{T}^{\cc N}(k)\rightleftarrows Sp_{T}^{\textrm{GL}}(k)$;
\item[$(2)$]
   $Sp_{T^2}^{\cc N}(k)\rightleftarrows Sp_{T^2}^{\textrm{SL}}(k)$;
\item[$(3)$]
   $Sp_{T^2}^{\cc N}(k)\rightleftarrows Sp_{T^2}^{\textrm{Sp}}(k)$;
\item[$(4)$]
   $Sp_{T^2}^{\cc N}(k)\rightleftarrows Sp_{T^2}^{\textrm{SO}}(k)\rightleftarrows Sp_{T^2}^{\textrm{O}}(k)$
if $\charr k\ne2$.
\end{itemize}
\end{thm}

We postpone its proof but first verify several statements which are
of independent interest. Recall that a motivic space $A\in\bb
M_\bullet$ is an {\it $\bb A^1$-$n$-connected\/} if the Nisnevich
sheaves $\pi_i^{\bb A^1}(A)\cong*$ for $i\leq n$. For any $B\in
SH(k)$, denote by $\pi_{i,n}^{\bb A^1}(B)$ the sheaf associated to
the presheaf
   $$U\in Sm/S\mapsto SH(k)(U_+\wedge S^{i-n}\wedge\gmpn,B).$$
$B$ is said to be {\it connected\/} if $\pi_{i,n}^{\bb A^1}(B)=0$ for $i<n$. We also set
   $$SH(k)_{\geq\ell}:=\Sigma^\ell_{S^1} SH(k)_{\geq 0}$$
and refer to the objects of $SH(k)_{\geq\ell}$ as {\it
$(\ell-1)$-connected}. We define the category of
$(\ell-1)$-connected $S^1$-spectra $SH_{S^1}(k)_{\geq\ell}$ in a
similar fashion. We say that a motivic space $A\in\bb M_\bullet$ is
{\it stably $(\ell-1)$-connected}, $\ell\geq 0$, if its suspension
$S^1$-spectrum is in $SH_{S^1}(k)_{\geq\ell}$ (i.e. all its negative
sheaves of stable homotopy groups are zero below $\ell$). Finally, a
motivic space $A\in\bb M_\bullet$ is {\it $(\ell-1)$-biconnected},
$\ell\geq 0$, if its suspension bispectrum (or its $\bb
P^1$-/$T$-spectrum) is in $SH(k)_{\geq\ell}$.

\begin{rem}
In the language of framed motives~\cite{GP1} if $A\in\bb
M_\bullet$ is $(\ell-1)$-biconnected and the base field is
(infinite) perfect then the framed motive $M_{fr}(A^c)$
(respectively the motivic space $C_*Fr(A^c)^{\textrm{gp}}$ with `gp' standing for group completion
of the sectionwise $H$-space $C_*Fr(A^c)$),
where $A^c$ is a cofibrant
resolution of $A$ in the projective model structure of spaces, is
locally $(\ell-1)$-connected as an $S^1$-spectrum (respectively as a
motivic space).
\end{rem}

It is well-known that the suspension bispectrum of a space is
connected. The following statement is a further extension of this
fact.\footnote{The author thanks A.~Ananyevskiy for pointing out a
helpful argument used in the proof of this proposition.}

\begin{prop}\label{connect}
Let $n>0$ and let $A\in\bb M_\bullet$ be an $\bb
A^1$-$(n-1)$-connected or stably $(n-1)$-connected pointed motivic
space. Then $A$ is $(n-1)$-biconnected.
\end{prop}

\begin{proof}
Let $A^f$ be a motivically fibrant replacement of $A$. First observe
that the suspension $S^1$-spectrum $\Sigma^\infty_{S^1}A^f$ is
locally $(n-1)$-connected. Indeed, the zeroth space of the spectrum
is locally $(n-1)$-connected by assumption, and hence each $m$th
space $A\wedge S^m$ of the spectrum is locally $(m+n-1)$-connected.
Morel's stable $\bb A^1$-connectivity theorem~\cite{Mor} implies
$\Sigma^\infty_{S^1}(A^f)$ is motivically $(n-1)$-connected.

Since $A^f$ is locally $(n-1)$-connected by assumption, it follows
that each $S^1$-spectrum $\Sigma^\infty_{S^1}(A\wedge\bb G_m^{\wedge
q})$, $q\geq 0$, is locally $(n-1)$-connected. By Morel's stable
$\bb A^1$-connectivity theorem~\cite{Mor} each
$\Sigma^\infty_{S^1}(A^f\wedge\bb G_m^{\wedge q})$, $q\geq 0$, is
motivically $(n-1)$-connected. Let $B=(B(0),B(1),\ldots)$ denote a
level motivically fibrant replacement of the bispectrum of
$\Sigma^\infty_{\gmp}\Sigma^\infty_{S^1}A^f$. Then each weight
$S^1$-spectrum $B(q)$ is motivically fibrant and locally
$(n-1)$-connected.

Since $B$ is a levelwise motivically fibrant bispectrum, then its
stabilization in the $\gmp$-direction $\Theta^\infty_{\gmp}B$ is
motivically fibrant. We have that $\Theta^\infty_{\gmp}B$ is a
fibrant replacement of $\Sigma^\infty_{\gmp}\Sigma^\infty_{S^1}A$.
By definition, the $q$th weight $S^1$-spectrum
$\Theta^\infty_{\gmp}B(q)$ is the colimit of the sequence
   $$B(q)\to\Omega_{\gmp}B(q+1)\to\Omega_{\bb G_m^{\wedge 2}}B(q+2)\to\cdots$$
Since $B(q+i)$ is motivically fibrant locally $(n-1)$-connected
$S^1$-spectrum, then so is $\Omega_{\bb G_m^{\wedge i}}B(q+i)$
by~\cite[A.2]{GP2}. We see that each $\Theta^\infty_{\gmp}B(q)$ is
locally $(n-1)$-connected. Using~\cite[A.2]{GP2} this is enough to
conclude that $\Theta^\infty_{\gmp}B\in SH(k)_{\geq n}$, and hence
$\Sigma^\infty_{\gmp}\Sigma^\infty_{S^1}A\in SH(k)_{\geq n}$.

The proof for stably $(n-1)$-connected motivic spaces is similar to
that for $\bb A^1$-$(n-1)$-connected spaces.
\end{proof}

The proof of the preceding proposition also implies the following

\begin{cor}\label{connectcor}
Under the assumptions of Proposition~\ref{connect}
the space $A\wedge C$ is $(n-1)$-biconnected for any $C\in\bb M_{\bullet}$.
\end{cor}

The next result is crucial for proving Theorem~\ref{compar}.

\begin{thm}\label{compmaps}
%Suppose $S=\spec k$ with $k$ a perfect field.
Given a pointed motivic space $C\in\bb M_\bullet$, the following natural maps are all
stable motivic equivalences of ordinary motivic $T$- and $T^2$-spectra:
\begin{itemize}
\item[$(1)$] $\lambda_n:F_{n+1}(C\wedge T)\to F_nC$, where the shift desuspension
functors are defined by~\eqref{desgl} in Example~\ref{glsp} for GL-spectra;
\item[$(2)$] $\lambda_n:F_{2n+2}(C\wedge T^2)\to F_{2n}C$, where the shift desuspension
functors are defined by~\eqref{dessl} in Example~\ref{slsp} for SL-spectra;
\item[$(3)$] $\lambda_n:F_{2n+2}(C\wedge T^2)\to F_{2n}C$, where the shift desuspension
functors are defined by~\eqref{dessp} in Example~\ref{spsp} for symplectic spectra;
\item[$(4)$] $\lambda_n:F_{2n+2}(C\wedge T^2)\to F_{2n}C$, where the shift desuspension
functors are defined by~\eqref{deso} in Example~\ref{orthsp} for orthogonal spectra
provided that $\charr k\ne 2$;
\item[$(5)$] $\lambda_n:F_{2n+2}(C\wedge T^2)\to F_{2n}C$, where the shift desuspension
functors are defined by~\eqref{desso} in Example~\ref{sosp} for SO-spectra
provided that $\charr k\ne 2$.
\end{itemize}
\end{thm}

\begin{rem}\label{remgh}
If $G$ is a linear algebraic group over a field $k$, and $H$ is a
closed subgroup, then by $G/H$ we mean the unpointed Nisnevich sheaf
associated with the presheaf $U\mapsto G(U)/H(U)$. If $G$ and $H$
are smooth and all $H$-torsors are Zariski locally trivial, then the
sheaf $G/H$ is represented by a scheme (see~\cite[p.~275]{AF1}). If there is
no likelihood of confusion, we
shall denote the scheme by the same symbol $G/H$.
By~\cite[p.~275]{AF1} this happens, for example,
if $H=GL_n, SL_n$ or $Sp_{2n}$. In turn, if $\charr k\ne 2$ then it
is proved similarly to~\cite[3.1.9]{AHW2} that the torsors
$O_{2n+2k}\to O_{2n+2k}/O_{2n}$ and $SO_{2n+2k}\to
SO_{2n+2k}/SO_{2n}$, $k,n>0$, are Zariski locally trivial. Note that
the schemes $GL_{n+k}/GL_{n}$, $SL_{2n+2k}/SL_{2n}$,
$Sp_{2n+2k}/Sp_{2n}$, $O_{2n+2k}/O_{2n}$, $SO_{2n+2k}/SO_{2n}$,
$k,n>0$, are all smooth (see~\cite[2.3.1]{AHW2} and~\cite[p.~5]{AHW3}). For the latter two
we assume $\charr k\ne 2$. As mentioned above, they all represent the
corresponding quotient Nisnevich sheaves (see~\cite[2.3.1]{AHW2} as well).
\end{rem}

\begin{proof}[Proof of Theorem~\ref{compmaps}]
(1). This is the case of GL-motivic spectra. By definition
(see~\eqref{desgl}),
   $$(F_nC)(q)=(GL_{q})_{+}\wedge_{GL_{q-n}}(C\wedge T^{q-n}).$$
For $q\geq n+1$, $\lambda_n(q)$ is the canonical quotient map
   $$(GL_{q})_{+}\wedge_{GL_{q-n-1}}(C\wedge T\wedge T^{q-n-1})=(GL_{q})_{+}\wedge_{GL_{q-n-1}}(C\wedge T^{q-n})
   \to (GL_{q})_{+}\wedge_{GL_{q-n}}(C\wedge T^{q-n}).$$
Since $T^n\wedge-$ reflects stable motivic equivalences of ordinary
$T$-spectra by~\cite[3.18]{Jar}, our statement reduces to showing that
$T^n\wedge\lambda_n$ is a stable motivic equivalence in $Sp_T^{\cc
N}(k)$.

The map $T^n\wedge\lambda_n$ takes the form
   $$(GL_{q})_{+}\wedge_{GL_{q-n-1}}(C\wedge T^{q})\to (GL_{q})_{+}\wedge_{GL_{q-n}}(C\wedge T^{q})$$
Since $GL_q$ acts on $T^q$, it follows from~\cite[1.2]{GM} that the
latter map is isomorphic to the map
   $$\lambda'_q:(GL_{q}/{GL_{q-n-1}})_+\wedge C\wedge T^{q}\to (GL_{q}/{GL_{q-n}})_+\wedge C\wedge T^{q}.$$
Here $GL_{q}/{GL_{q-n-1}},GL_{q}/{GL_{q-n}}$ are smooth schemes of
Remark~\ref{remgh}. Set,
   \begin{multline*}
    F_{n+1}'(C\wedge T):=(*,\bl n\ldots,*,(GL_{n+1})_+\wedge C\wedge T^{n+1},(GL_{n+2}/{GL_{1}})_+\wedge C\wedge T^{n+2},\\
       (GL_{n+3}/{GL_{2}})_+\wedge C\wedge T^{n+3},\ldots)
   \end{multline*}
and
   \begin{multline*}
    F_{n}'(C):=(*,\bl{n-1}\ldots,*,(GL_{n})_{+}\wedge C\wedge T^n,(GL_{n+1}/GL_1)_+\wedge C\wedge T^{n+1},\\
    (GL_{n+2}/{GL_{2}})_+\wedge C\wedge T^{n+2},
       (GL_{n+3}/{GL_{3}})_+\wedge C\wedge T^{n+3},\ldots).
    \end{multline*}
\normalsize The structure of $T$-spectra on $F_{n+1}'(C\wedge T)$ and $F_{n}'(C)$
are obvious. It is induced by the action of $T$ on the right.

Consider a commutative diagram of ordinary motivic $T$-spectra
   $$\xymatrix{sh^{-n-1}(\Sigma_T^\infty (C\wedge T\wedge T^n))\ar[d]\ar[r]^(.6)\alpha&F_{n+1}'(C\wedge T)\ar[d]^{\lambda'}\\
               sh^{-n}(\Sigma_T^\infty (C\wedge T^n))\ar[r]_{\beta}&F_n'(C),}$$
where $sh^{-n}(\Sigma_T^\infty(C\wedge
T^n))=(*,\bl{n-1}\ldots,*,C\wedge T^n,C\wedge T^{n+1},\ldots)$ is
the $(-n)$th shift of $\Sigma_T^\infty C$, $\alpha,\beta$ are
induced by the the following injective maps in $\bb M_\bullet$:
   $$S^0\to (GL_{q}/{GL_{q-n-1}})_+,\quad S^0\to (GL_{q}/{GL_{q-n}})_+.$$
They send the basepoint of $S^0$ to $+$ and the unbasepoint to
$GL_{q-n-1}$ and $GL_{q-n}$ respectively. Note that the left
vertical arrow is a stable motivic equivalence in $Sp^{\cc N}_T(k)$.
Observe that $\alpha$ and $\beta$ are isomorphic to counit andjunction maps
   $$T^n\wedge F_{n+1}^{\cc N}(C\wedge T)\to T^n\wedge F_{n+1}^{GL}(C\wedge T),\quad
       T^n\wedge F_{n}^{\cc N}(C)\to T^n\wedge F_{n}^{GL}(C).$$
To show that $T^n\wedge\lambda_n$ is a stable motivic equivalence it
is enough to show that $\alpha$ and $\beta$ are stable motivic equivalences.

The map $\alpha$ fits in a level cofiber sequence of $T$-spectra
   $$sh^{-n-1}(\Sigma_T^\infty (C\wedge T\wedge T^n))\hookrightarrow F_{n+1}'(C\wedge T)\to F_{n+1}''(C\wedge T),$$
where $F_{n+1}''(C\wedge T)$ is the bispectrum
   $$(*,\bl n\ldots,*,GL_{n+1}\wedge C\wedge T^{n+1},(GL_{n+2}/{GL_{1}})\wedge C\wedge T^{n+2},
       (GL_{n+3}/{GL_{2}})\wedge C\wedge T^{n+3},\ldots),$$
where $GL_{n+1}$ is pointed at the identity matrix and
$GL_{n+i}/{GL_{i-1}}$ is pointed at $GL_{i-1}$.

We claim that $F_{n+1}''(C\wedge T)$ is isomorphic to zero in
$SH(k)$. This is equivalent to saying that
   $$F''(C\wedge T):=(GL_{n+1}\wedge C,(GL_{n+2}/{GL_{1}})\wedge C\wedge T,
       (GL_{n+3}/{GL_{2}})\wedge C\wedge T^{2},\ldots)$$
is isomorphic to zero in $SH(k)$ (we use here~\cite[3.18]{Jar}).
Every $T$-spectrum $E=(E_0,E_1,\ldots)$ has the layer filtration
$E=\colim_iL_iE$ with $L_nE=(E_0,\ldots,E_{i-1},E_i,E_i\wedge
T,E_i\wedge T^2,\ldots)$. The $i$th layer $L_iF''(C\wedge T)$ of
$F''(C\wedge T)$ is isomorphic to
$\Sigma^\infty_T((GL_{n+1+i}/GL_i)\wedge C)$ in $SH(k)$.

By the proof of~\cite[2.1.3]{AFW} the ``projection onto the first
column map" $GL_n/GL_{n-1}\to\bb A^n\setminus 0$ is a motivic
equivalence of spaces.
%By~\cite[3.2.20]{MV} $\bb A^n\setminus 0$ is
%motivically equivalent to $S^{n-1}\wedge\gmpn$ for $n\geq 1$.
%Since the base field is perfect by assumption, it follows from
%Morel's unstable $\bb A^1$-connectivity theorem~\cite[6.38]{Mor1}
%(see~\cite[2.2.12]{AWW} as well)
It follows from~\cite[2.1.4]{AFW} that $\bb A^n\setminus 0$ is $\bb
A^1$-$(n-2)$-connected for $n\geq 2$, and hence so is
$GL_n/GL_{n-1}$. If we consider a fibre sequence of motivic spaces
   $$GL_{n+l-1}/GL_{n-1}\to GL_{n+k+l-1}/GL_{n-1}\to GL_{n+k+l-1}/GL_{n+l-1},\quad k,l\geq 0,$$
we conclude by induction that $GL_{n+k}/GL_{n-1}$ is $\bb
A^1$-$(n-2)$-connected as well for $n\geq 2$.

By Corollary~\ref{connectcor}
$\Sigma^\infty_T((GL_{n+i+1}/GL_{i})\wedge C)\in SH(k)_{\geq i-1}$
if $i\geq 2$, and hence $L_iF''(C\wedge T)\in SH(k)_{\geq i-1}$. We
see that $F''(C\wedge T)\in\bigcap_{i\in\bb N} SH(k)_{\geq i}$. This
is only possible when $F''(C\wedge T)\cong 0$ in $SH(k)$, and our
claim follows. Thus $\alpha$ is a stable equivalence, because its
cofiber $F_{n+1}''(C\wedge T)$ is zero in $SH(k)$. Using the same
arguments, $\beta$ is a stable equivalence as well, and hence so is
$\lambda_n$ as stated.

The proof of (2), (3), (4), (5) literally repeats that of (1) if we
use~\cite[2.13]{AHW3} saying that $SL_n/SL_{n-1}$,
$Sp_{2m}/Sp_{2m-2}$ (with $n=2m$) are isomorphic to the
odd-dimensional motivic sphere $Q_{2n-1}$ which is, by definition,
the affine quadric defined by the equation $\sum_{i=1}^n x_iy_i=1$.
By~\cite[p.~1892]{ADF} $Q_{2n-1}$ is $\bb A^1$-equivalent to $\bb
A^n\setminus 0$, and hence it is $\bb A^1$-$(n-2)$-connected for
$n\geq 2$ by~\cite[2.1.4]{AFW}. Likewise, $SO_{2n+1}/SO_{2n}$
(hence $O_{2n+1}/O_{2n}$ as well) is isomorphic to $Q_{2n}$ by the proof
of~\cite[2.15]{AHW3}, where $Q_{2n}$ is the affine quadric defined by the
equation $\sum_{i=1}^n x_iy_i=z(1-z)$. Since $Q_{2n}$ is $\bb A^1$-equivalent to
$S^n\wedge\bb G_m^{\wedge n}$ by~\cite[Theorem~2]{ADF}, it is
$\bb A^1$-$(n-1)$-connected for $n\geq 1$. Also, $SO_{2n}/SO_{2n-1}$
(hence $O_{2n}/O_{2n-1}$ as well) is isomorphic to $Q_{2n-1}$ by~\cite[2.13]{AHW3},
and so it is $\bb A^1$-$(n-2)$-connected for $n\geq 2$.
\end{proof}

We shall need the following useful fact.

\begin{prop}\label{weakeqs}
%Suppose $S=\spec k$ with $k$ a perfect field and
Let $G\in\{\textrm{GL,SL,Sp,O,SO}\}$. A map $f:X\to Y$ of $G$-spectra in
the sense of Examples~\ref{glsp}-\ref{sosp} is a stable equivalence
in the sense of Definition~\ref{modelst} if and only if it is a
stable motivic equivalence of ordinary motivic spectra.
\end{prop}

\begin{proof}
We prove the statement for GL-motivic $T$-spectra, because the proof
for the other cases is similar. Denote by
   $$\wt{\cc P}^{\cc N}:=\{cyl(\lambda_n:F_{n+1}^{\cc N}(C\wedge T)\to F_{n}^{\cc N}T)\mid\lambda_n\in\cc P\},$$
where $cyl$ refers to the ordinary mapping cylinder map, $\cc P$ is
the family of Definition~\ref{modelst} corresponding to ordinary
$T$-spectra. Similarly, set
   $$\wt{\cc P}^{GL}:=\{cyl(\lambda_n:F_{n+1}^{\textrm{GL}}(C\wedge T)\to F_{n}^{\textrm{GL}}T)\mid\lambda_n\in\cc P\},$$
where $\cc P$ is the family of Definition~\ref{modelst}
corresponding to GL-spectra. Then $\wt{\cc P}^{\cc N}$ (respectively
$\wt{\cc P}^{GL}$) is a family of cofibrations in $Sp_T^{\cc N}(k)$
(respectively in $Sp_T^{\textrm{GL}}(k)$) with respect to the stable
model structure. Also, the left Kan extension functor of
Corollary~\ref{adjunc}
   $$L:Sp_T^{\cc N}(k)\to Sp_T^{\textrm{GL}}(k)$$
takes $\wt{\cc P}^{\cc N}$ to $\wt{\cc P}^{GL}$.

The proof of Theorem~\ref{compmaps} shows that the commutative
square in $Sp_T^{\cc N}(k)$
   $$\xymatrix{F_{n+1}^{\cc N}(C\wedge T)\ar[d]_{\alpha}\ar[r]^{\lambda_n}&F_{n}^{\cc N}C\ar[d]^{\beta}\\
               F_{n+1}^{\textrm{GL}}(C\wedge T)\ar[r]^{\lambda_n}&F_{n}^{\textrm{GL}}C}$$
with vertical maps being the counit maps consists of stable motivic
equivalences. Since the cylinder maps are preserved by the forgetful functor
   $$U:Sp_T^{\textrm{GL}}(k)\to Sp_T^{\cc N}(k),$$
it follows that $U(\wt{\cc P}^{GL})$ is a family of injective stable
motivic equivalences.

Let $J$ be a family of generating trivial flasque
cofibrations~\cite[3.2(b)]{Is} for $\bb M_\bullet$.
By~\cite[3.10]{Is} domains and codomains of the maps in $J$ are
finitely presentable. Recall that the set of maps in $Sp_T^{\cc
N}(k)$ (respectively in $Sp_T^{\textrm{GL}}(k)$)
   $$\cc P_J^{\cc N}:=\bigcup_{n\geq 0}F_n^{\cc
   N}(J)\quad(\textrm{respectively}\ \cc P_J^{\textrm{GL}}:=\bigcup_{n\geq 0}F_n^{\textrm{GL}}(J))$$
is a family of generating trivial cofibrations for the pointwise
model structure of Proposition~\ref{pmodel} (see, e.g., the proof
of~\cite[4.2]{DRO}). By construction, $L(\cc P_J^{\cc N})=\cc
P_J^{\textrm{GL}}$.

We set
   $$\wh{\cc P^{\textrm{GL}}}:=\{A\wedge\Delta[n]_+\sqcup_{A\wedge\partial\Delta[n]_+}B\wedge\partial\Delta[n]_+
     \to B\wedge\Delta[n]_+\mid(A\to B)\in\wt{\cc P}^{\textrm{GL}},n\geq 0\}.$$
An augmented family of $\wt{\cc P}^{\textrm{GL}}$-horns is the
following family of trivial cofibrations:
   $$\Lambda({\cc P}^{\textrm{GL}})=\cc P_J^{\textrm{GL}}\cup\wh{{\cc P}^{\textrm{GL}}}.$$
Observe that domains and codomains of the maps in $\Lambda({\cc P}^{\textrm{GL}})$
are finitely presentable. It can be proven similarly
to~\cite[4.2]{H} that a map $f:A\to B$ is a fibration in the stable
model structure with fibrant codomain if and only if it has the
right lifting property with respect to $\Lambda({\cc
P}^{\textrm{GL}})$.

By~\cite[2.12]{Jar} a map $f:X\to Y$ in $Sp_T^{\cc N}(k)$ is a
stable motivic equivalence if and only if it induces a weak
equivalence $f^*:\Map_*(Y,W)\to\Map_*(X,W)$ of Kan complexes for all
stably fibrant injective $T$-spectra $W$. It follows that a pushout
of an injective stable motivic equivalence is an injective stable
motivic equivalence. Since all colimits in $Sp_T^{\textrm{GL}}(k)$
are computed in $Sp_T^{\cc N}(k)$, it follows that a pushout of a
coproduct of maps from $\Lambda({\cc P}^{\textrm{GL}})$ computed in
$Sp_T^{\textrm{GL}}(k)$ is a stable motivic equivalence in
$Sp_T^{\cc N}(k)$, because every map of $\Lambda({\cc
P}^{\textrm{GL}})$ is an injective stable motivic equivalence in
$Sp_T^{\cc N}(k)$. In particular, $U$ sends $\Lambda({\cc
P}^{\textrm{GL}})$-cell complexes to stable motivic equivalence in
$Sp_T^{\cc N}(k)$.

We now apply the small object argument to the family $\Lambda({\cc
P}^{\textrm{GL}})$ in order to fit $f:X\to Y$ of the proposition
into a commutative diagram
   $$\xymatrix{U(X)\ar[r]\ar[d]_{U(f)}&U(L_{\cc P}X)\ar[d]^{U(L_{\cc P}f)}\\
               U(Y)\ar[r]&U(L_{\cc P}Y)}$$
with $X\to L_{\cc P}X$, $Y\to L_{\cc P}Y$ being $\Lambda({\cc
P}^{\textrm{GL}})$-cell complexes and $L_{\cc P}X,L_{\cc P}Y$ stably
fibrant GL-spectra (hence stably fibrant ordinary spectra by
Corollary~\ref{adjuncstcor}). Notice that $L_{\cc P}f$ is a level
motivic equivalence. Our statement now follows.
\end{proof}

\begin{proof}[Proof of Theorem~\ref{compar}]
We only prove that
   $$L:Sp_{T}^{\cc N}(k)\rightleftarrows Sp_{T}^{\textrm{GL}}(k):U$$
is a Quillen equivalence with respect to the stable model structure,
because the other cases are proved in a similar fashion.

The proof of Theorem~\ref{compmaps} shows that the counit map
$\beta_n:F_n^{\cc N}(C)\to U(F_n^{\textrm{GL}}(C))$ is a stable
motivic equivalence in $Sp_T^{\cc N}(k)$ for any $C\in\bb
M_\bullet$. Suppose $E\in Sp_T^{\cc N}(k)$ is cofibrant. Present it
as $E=\colim_n L_n^{\cc N}(E)$, where each layer $L_k^{\cc
N}(E)=(E_0,\ldots,E_{n-1},E_n,E_n\wedge T,\ldots)$ is cofibrant as
well. The canonical map $\phi_n:F_n^{\cc N}(E_n)\to L_n^{\cc N}(E)$
is a stable motivic equivalence of cofibrant objects.

Denote by $L_n^{\textrm{GL}}(E):=L(L_n^{\cc N}(E))$. Then
$L(E)=\colim_n L_n^{\textrm{GL}}(E)$, because $L$ preserves
colimits. By Corollary~\ref{adjuncstcor} $L$ is a left Quillen
functor, and hence $L(\phi_n):F_n^{\textrm{GL}}(E_n)\to
L_n^{\textrm{GL}}(E)$ is a stable equivalence in
$Sp_T^{\textrm{GL}}(k)$ by~\cite[1.1.12]{Hov}. By Proposition~\ref{weakeqs} $UL(\phi_n)$
is a stable motivic equivalence in $Sp_T^{\cc N}(k)$. Consider a
commutative square
   $$\xymatrix{F_n^{\cc N}(E_n)\ar[d]_{\phi_n}\ar[r]^{\beta_n}&U(F_n^{\textrm{GL}}(E_n))\ar[d]^{UL(\phi_n)}\\
               L_n^{\cc N}(E)\ar[r]_{\gamma_n}&U(L_n^{\textrm{GL}}(E))}$$
with $\gamma_n$ the counit map. Since $\beta_n,\phi_n,UL(\phi_n)$
are stable motivic equivalences, then so is $\gamma_n$. It follows
from~\cite[3.5]{DRO} that $\gamma:=\colim_n\gamma_n:E=\colim_n
L_n^{\cc N}(E)\to UL(E)=\colim_n U(L_n^{\textrm{GL}}(E))$ is a
stable motivic equivalence, because $Sp_T^{\cc N}(k)$ is a weakly
finitely generated model category.

Let $\delta:L(E)\to RL(E)$ be a fibrant resolution of $L(E)$ in
$Sp_T^{\textrm{GL}}(k)$. Then $U(\delta)$ is a stable motivic
equivalence in $Sp_T^{\cc N}(k)$ by Proposition~\ref{weakeqs}. We
see that the composition
$E\xrightarrow{\gamma}UL(E)\xrightarrow{U(\delta)}U(RL(E))$ is a
stable motivic equivalence for any cofibrant $E\in Sp_T^{\cc N}(k)$.
Since $U$ plainly reflects stable equivalences between fibrant
GL-spectra, $(L,U)$ is a Quillen equivalence by~\cite[1.3.16]{Hov}.
This completes the proof of the theorem.
\end{proof}

We discuss an application of Theorem~\ref{compar} in the next
section concerning the localization functor $C_*\cc Fr$
of~\cite{GP5}.

\section{On the localization functor $C_*\cc Fr$}\label{locfun}

Throughout this section $k$ is an (infinite) perfect
field. As usual, we assume $\charr k\not=2$ whenever we deal with
orthogonal or special orthogonal motivic spectra. Recall that $SH_{\nis}(k)$
is the triangulated category obtained from
the local stable homotopy category of sheaves of $S^1$-spectra
$SH_{S^1}^{\nis}(k)$ by stabilizing $SH_{S^1}^{\nis}(k)$ with
respect to the endofunctor $\gmp\wedge-$.

Let $\cc T$ be a triangulated category. Following~\cite{AJS}, we define a localization in
$\cc T$ as a triangulated endofunctor $L:\cc T\to\cc T$ together
with a natural transformation $\eta:\id\to L$ such that
$L\eta_X=\eta_{LX}$ for any $X$ in $\cc T$ and $\eta$ induces an
isomorphism $LX\cong LLX$. We refer to $L$ as a {\it localization
functor in $\cc T$}. Such a localization functor determines a full
subcategory $\kr L$ whose objects are those $X$ such that $LX=0$.
An object $X\in\cc T$ is said to be {\it $L$-local\/} if $\eta_X:X\to LX$
is an isomorphism.

The computation of localization functors and their full subcategories
of local objects is enormously hard in practice. In particular, if
$\cc T=SH_{\nis}(k)$ and $\cc S$ is the full subcategory of $SH_{\nis}(k)$ compactly generated
by the shifted cones of the arrows
$pr_X:\Sigma^\infty_{S^1}\Sigma^\infty_{\gmp}(X\times\bb
A^1)_+\to\Sigma^\infty_{S^1}\Sigma^\infty_{\gmp}X_+$, $X\in Sm/k$, then the Bousfield
localization theory in compactly generated triangulated categories says that
there exists a localisation functor
   $$L_{\bb A^1}:SH_{\nis}(k)\to SH_{\nis}(k)$$
such that $\cc S=\kr L_{\bb A^1}$.
%The functor $L_{\bb A^1}$ is defined as the homotopy colimit
%of an infinite tower in $SH_{\nis}(k)$ which is reminiscent to the small object argument.
By definition, the Morel--Voevodsky stable motivic
homotopy category $SH(k)$ is the quotient category $SH_{\nis}(k)/\cc S$.

A new approach to
the classical stable homotopy theory $SH(k)$ of Morel--Voevodsky~\cite{MV} was
suggested in~\cite{GP5}. This approach has nothing to do with any kind of
motivic equivalences and is briefly defined as follows. There exists an
explicit localization functor
   $$C_*\cc Fr:SH_{\nis}(k)\to SH_{\nis}(k)$$
that first takes a bispectrum $E$ to its naive projective cofibrant
resolution $E^c$ and then one sets in each bidegree $C_*\cc
Fr(E)_{i,j}:=C_*Fr(E^c_{i,j})$ (we refer the reader to~\cite{GP1}
for the definition of $C_*Fr(\cc X)$, $\cc X\in\bb M_\bullet$). We
should note that the localization functor $C_*\cc Fr$ is isomorphic
to the big framed motives localization functor $\cc M^b_{fr}$
of~\cite{GP1} (see~\cite{GP5} as well). We then define $SH^{new}(k)$
as the category of $C_*\cc Fr$-local objects in $SH_{\nis}(k)$.
By~\cite[Section~2]{GP5} $SH^{new}(k)$ is canonically equivalent to
Morel--Voevodsky's $SH(k)$.

The localization functor $C_*\cc Fr$ is also of great utility when
dealing with another model for $SH(k)$, constructed in~\cite{GP5}.
This model recovers all motivic bispectra as certain covariant
functors on $Fr_0(k)$ taking values in $\bb A^1$-local framed
$S^1$-spectra. In particular, this model of $SH(k)$ implies that
$\pi_{i,j}^{\bb A^1}(E)$-s have more information than just the naive
bigraded sheaves. Namely, they are recovered from certain covariant
functors $\pi_{i}^{fr}(E)$ on $Fr_0(k)$ taking values in strictly
$\bb A^1$-invariant framed sheaves. Thus the functors
$\pi_{i}^{fr}(E)$ have one index only corresponding to the
$S^1$-direction (in this way we get rid of the second index).
These are reminiscent of the classical stable homotopy
groups of ordinary $S^1$-spectra.
It is therefore useful to think of the $\pi_{i,j}^{\bb A^1}(E)$ as
the richer information ``$\pi_{i}^{fr}(E)$".

Theorems~\ref{compar} and~\ref{compmaps} give rise to an equivalent
model for the localization functor $C_*\cc Fr$ (see below). It
involves smooth algebraic varieties of the form $G_{n+k}/G_n$, where
$G_n$, $n\geq 0$, is $GL_n$, $SL_{2n}$, $Sp_{2n}$, $O_{2n}$ or $SO_{2n}$.
Below we shall write $\cc G$ to denote the family $\{G_n\}_{n\geq
0}$. In this paper $\cc G$ is $\{GL_n\}_{n\geq 0}$, $\{SL_{2n}\}_{n\geq
0}$, $\{Sp_{2n}\}_{n\geq 0}$, $\{O_{2n}\}_{n\geq 0}$ or $\{SO_{2n}\}_{n\geq
0}$.

\begin{dfn}\label{deffrmot}
Let $\cc G=\{G_k\}_{k\geq 0}$ be a family as above, $n\geq 0$ and
$\cc X\in\bb M_\bullet$. If $\cc G=\{GL_k\}_{k\geq 0}$ define
   $$Fr^{\cc G,n}(\cc X):=\colim_{q\geq n}\uhom_{\bb M_\bullet}(\bb P^{\wedge q},
     \cc X\wedge(G_q/G_{q-n})_+\wedge T^q).$$
In other words, if we consider the $\bb P^1$-spectrum
   $$\cc Y:=(*,\bl{n-1}\ldots,*,\cc X\wedge (G_{n})_{+}\wedge T^n,\cc X\wedge(G_{n+1}/G_{1})_+\wedge
     T^{n+1},\cc X\wedge(G_{n+2}/G_{2})_+\wedge T^{n+1},\ldots)$$
then $Fr^{\cc G,n}(\cc X)$ equals the $0$th space of the spectrum
$\Theta^\infty_{\bb P^1}(\cc Y)$. Notice that $G_{n+k}/G_n$-s
incorporated into the definition are all smooth algebraic varieties. In turn,
if $\cc G$ is $\{SL_{2k}\}_{k\geq 0}$, $\{Sp_{2k}\}_{k\geq 0}$,
$\{O_{2k}\}_{k\geq 0}$ or $\{SO_{2k}\}_{k\geq 0}$ and $n\geq 0$ is even,
then $Fr^{\cc G,n}(\cc X)$ is defined as above if we take the colimit over even $q$-s.

Using the terminology of~\cite{GP1}, we define the {\it $(\cc
G,n)$-framed motive $M_{fr}^{\cc G,n}(\cc X)$ of $\cc X$\/} as the
Segal $S^1$-spectrum associated with the (sectionwise) $\Gamma$-space
$m\in\Gamma^{\op}\mapsto C_*Fr^{\cc G,n}(\cc X\wedge m_+)$,
where $C_*$ stands for the Suslin complex.

If we want to specify the choice of groups, we write below
$C_*Fr^{GL,n}(\cc X)$, $C_*Fr^{SL,2n}(\cc X)$, $C_*Fr^{Sp,2n}(\cc X)$,
$C_*Fr^{O,2n}(\cc X)$, and $C_*Fr^{SO,2n}(\cc X)$ (respectively, we write
$M_{fr}^{GL,n}(\cc X)$, $M_{fr}^{SL,2n}(\cc X)$, $M_{fr}^{Sp,2n}(\cc
X)$, $M_{fr}^{O,2n}(\cc X)$, and $M_{fr}^{SO,2n}(\cc X)$).
\end{dfn}

%Recall from~\cite{GP1} the definition of the category $SmOp(Fr_0(k))$. Its objects are
%pairs $(X,U)$, where $X\in Sm/k$ and $U\subset X$ is an open
%subset. A morphism between $(X,U)$ and $(X',U')$ in $SmOp(Fr_0(k))$
%is a morphism $f\in Fr_0(X,X')$ such that $f(U)\subset U'$. We
%identify $X\in Sm/k$ with the pair $(X,\emptyset)\in SmOp(Fr_0(k))$.
Let $\Delta^{\op} Fr_0(k)$ be the category of simplicial
objects in $Fr_0(k)$. There is an obvious fully faithful functor $spc:
Fr_0(k) \to Shv_\bullet(Sm/k)$ sending an object $X\in
Fr_0(k)$ to the Nisnevich sheaf $X_+$. It induces a fully faithful functor
   $$spc:\Delta^{\op} Fr_0(k)\to sShv_\bullet(Sm/k),$$
taking an object $[n]\mapsto Y_n$ to the simplicial Nisnevich
sheaf $[n]\mapsto (Y_n)_+$. Denote the image of this functor by
$\cc T$. Also, we shall write $\bl{\to}{\cc T}$ to denote the
motivic spaces which are filtered colimits of objects in $\cc T$
coming from filtered diagrams in $\Delta^{\op} Fr_0(k)$ under
the functor $spc$.

\begin{thm}\label{frmotgn}
Suppose $\cc X\in\bl{\to}{\cc T}$. Under the notation of
Definition~\ref{deffrmot} there is a natural stable local
equivalence of $S^1$-spectra $\mu:M_{fr}(\cc X)\to M_{fr}^{GL,n}(\cc
X)$, where $n\geq 0$. If $n$ is even and $\cc G\in\{SL,Sp,O,SO\}$
then there is also a natural stable local equivalence of
$S^1$-spectra $\mu:M_{fr}(\cc X)\to M_{fr}^{\cc G,n}(\cc X)$.
\end{thm}

\begin{proof}
We shall prove the theorem for the case $\cc G=\{GL_n\}_{n\geq 0}$.
The proof for the other choices of $\cc G$ is similar. Without loss
of generality we may assume for simplicity $\cc X=X_+$, where $X\in
Sm/k$. By the proof of
Theorem~\ref{compmaps} there is a natural stable motivic equivalence
of $T$-spectra
   $$\beta:sh^{-n}(\Sigma_T^\infty (\cc X\wedge T^n))\to\cc Y,$$
where $sh^{-n}(\Sigma_T^\infty(\cc X\wedge
T^n))=(*,\bl{n-1}\ldots,*,\cc X\wedge T^n,\cc X\wedge
T^{n+1},\ldots)$ is the $(-n)$th shift of $\Sigma_T^\infty(\cc X\wedge T^n)$ and
$\cc Y$ as in Definition~\ref{deffrmot}. Observe that both spectra
are Thom spectra with the bounding constant $d\leq 1$ in the sense
of~\cite{GN}.

By the proof of~\cite[2.13]{Jar} $\beta$ is a stable motivic
equivalence of $\bb P^1$-spectra, and hence so is
   $$\Theta^\infty_{\bb P^1}(\beta):\Theta^\infty_{\bb P^1}(sh^{-n}(\Sigma_T^\infty (\cc X\wedge T^n)))
     \to\Theta^\infty_{\bb P^1}(\cc Y).$$
We have,
   $$C_*\Theta^\infty_{\bb P^1}(sh^{-n}(\Sigma_T^\infty (\cc X\wedge T^n)))
       =(C_*Fr(\cc X),C_*Fr(\cc X\wedge T),C_*Fr(\cc X\wedge T^2),\ldots)$$
and
   $$C_*\Theta^\infty_{\bb P^1}(\cc Y)
       =(C_*Fr^{GL,n}(\cc X),C_*Fr^{GL,n}(\cc X\wedge T),C_*Fr^{GL,n}(\cc X\wedge T^2),\ldots).$$
Since the map $C_*\Theta^\infty_{\bb P^1}(sh^{-n}(\Sigma_T^\infty
(\cc X\wedge T^n)))\to C_*\Theta^\infty_{\bb P^1}(\cc Y)$ is a
stable motivic equivalence, it follows from~\cite[5.2]{GN} that the
map of spaces
   $$\nu:C_*Fr(\cc X\wedge T)\to C_*Fr^{GL,n}(\cc X\wedge T)$$
is a local equivalence. By~\cite[A.1]{GNP} and the proof of~\cite[9.9]{GN} both
spaces are locally connected. It follows from~\cite[6.4]{GP1}
that these are the underlying spaces of
(locally) very special $\Gamma$-spaces, and so the map of
$S^1$-spectra
   $$\xi:M_{fr}(\cc X\wedge T)\to M_{fr}^{GL,n}(\cc X\wedge T)$$
is a level local equivalence.

Consider a commutative diagram
   $$\xymatrix{M_{fr}(\cc X)_f\ar[r]\ar[d]&M_{fr}^{GL,n}(\cc X)_f\ar[d]\\
               \Omega_{\bb P^1}(M_{fr}(\cc X\wedge T)_f)\ar[r]^{\xi_*}
               &\Omega_{\bb P^1}(M_{fr}^{GL,n}(\cc X\wedge T)_f).}$$
Here $f$ refers to the stable local fibrant replacement of
$S^1$-spectra and the upper arrow is induced by $\beta$. It follows
from~\cite[7.1]{GP1} that all spectra are motivically fibrant. Then
the map $\xi_*$ is a level weak equivalence of motivically fibrant
spectra. The proof of~\cite[4.1(2)]{GP1} shows that the vertical
arrows are level weak equivalences (we also
use~\cite[Section~9]{GN}), and hence so is the upper arrow. It
follows that the map
   $$M_{fr}(\cc X)\to M_{fr}^{GL,n}(\cc X)$$
is a stable local equivalence, as was to be shown.
\end{proof}

If $\cc X\mapsto\cc X^c$ is the cofibrant replacement functor in the
projective motivic model structure in $\bb M_\bullet$, then
$\cc X^c$ belongs to $\bl{\to}{\cc T}$ (see~\cite[Section~10]{GP1}).

\begin{thm}\label{prilozh}
Under the assumptions of Theorem~\ref{frmotgn} let $C_*\cc Fr^{\cc
G,n}$ be the functor on bispectra taking an $(S^1,\bb G_m^{\wedge
1})$-bispectrum $E$ to the bispectrum $C_*\cc Fr^{\cc G,n}(E)$ which
is defined in each bidegree as $C_*\cc Fr^{\cc
G,n}(E)_{i,j}:=C_*Fr^{\cc G,n}(E^c_{i,j})$, where $E^c$ is a
projective cofibrant resolution of $E$. Then $C_*\cc Fr^{\cc G,n}$
is an endofunctor on $SH_{\nis}(k)$ and is naturally isomorphic to the localizing functor
   $$C_*\cc Fr:SH_{\nis}(k)\to SH_{\nis}(k)$$
if $\cc G=\{GL_k\}_{k\geq 0}$ and $n$ is any non-negative integer,
or if $\cc G\in\{SL,Sp,O,SO\}$ and $n$ is even non-negative. In
particular, one has a localizing functor
   $$C_*\cc Fr^{\cc G,n}:SH_{\nis}(k)\to SH_{\nis}(k)$$
such that the category of $C_*\cc Fr^{\cc G,n}$-local objects is
$SH^{new}(k)$.
\end{thm}

\begin{proof}
By the Additivity Theorem of~\cite{GP1} $C_*\cc Fr(-,Y)$ and $C_*\cc
Fr^{\cc G,n}(-,Y)$ are special $\Gamma$-spaces for $Y$ a filtered
colimit of simplicial schemes from $\Delta^{\op}Fr_0(k)$. Let $F$ be
an $S^1$-spectrum such that every entry $F_j$ of $F$ is a filtered
colimit of $k$-smooth simplicial schemes from $\Delta^{\op}Fr_0(k)$.
$F$ has a natural filtration $F=\colim_m L_m(F)$, where $L_m(F)$ is
the spectrum
   $$(F_{0},F_{1},\ldots,F_{m},F_{m}\wedge S^1,F_{m}\wedge S^2,\ldots).$$
Then $C_*Fr(F)=C_*Fr(\colim_m L_m(F))=\colim_m C_*Fr(L_m(F))$, where
$C_*Fr(L_m(F))$ is the spectrum
   $(C_*Fr(F_{0}),C_*Fr(F_{1}),\ldots,C_*Fr(F_{m}),C_*Fr(F_{m}\otimes S^1),C_*Fr(F_{m}\otimes S^2),\ldots).$
Similarly, one has $C_*Fr^{\cc G,n}(F)=C_*Fr^{\cc G,n}(\colim_m
L_m(F))=\colim_mC_*Fr^{\cc G,n}(L_m(F))$, where $C_*Fr^{\cc
G,n}(L_m(F))$ is the spectrum
   $$(C_*Fr^{\cc G,n}(F_{0}),C_*Fr^{\cc
      G,n}(F_{1}),\ldots,C_*Fr^{\cc G,n}(F_{m}),C_*Fr^{\cc
      G,n}(F_{m}\otimes S^1),C_*Fr^{\cc G,n}(F_{m}\otimes S^2),\ldots).$$
Observe that $sh^n C_*Fr(L_m(F))=M_{fr}(F_m)$ and $sh^n C_*Fr^{\cc
G,n}(L_m(F))=M_{fr}^{\cc G,n}(F_m)$.

By Theorem~\ref{frmotgn} the natural map $M_{fr}(F_m)\to M_{fr}^{\cc
G,n}(F_m)$ is a stable local equivalence, and hence so is
$C_*Fr(L_m(F))\to C_*Fr^{\cc G,n}(L_m(F))$. Thus the natural
map $C_*Fr(F)\to C_*Fr^{\cc G,n}(F)$ is a stable local
equivalence of spectra. Thus if $E$ is a bispectrum then the natural map of
bispectra $C_*\cc Fr(E)\to C_*\cc Fr^{\cc G,n}(E)$ is a level stable
local equivalence. The fact that $C_*\cc Fr^{\cc G,n}$
is an endofunctor on $SH_{\nis}(k)$ is obvious as well as that both functors are
isomorphic on $SH_{\nis}(k)$. This completes the proof.
\end{proof}

%As a consequence of the theorem, we can incorporate linear algebraic groups into
%the theory of motivic infinite loop spaces and framed motives developed in~\cite{GP1}.

%\bibliographystyle{amsalpha}

\end{document}